\documentclass[number,citesort,dvips]{arxbj}
\usepackage{upgreek}

\aid{0}
\volume{17}
\issue{1}
\pubyear{2011}
\firstpage{194}
\lastpage{210}
\doi{10.3150/10-BEJ266}

\makeatletter
\DeclareMathAlphabet\mathcaligr{OMS}{cmsy}{m}{n}
\renewcommand{\mathcal}{\mathcaligr}
\newcommand{\cal}{\mathcal}
\newcommand{\bb}{\mathbb}
\def\R{{\bb R}}
\def\E{{\bb E}}
\def\var{\operatorname{\bb Var}}
\def\cov{\operatorname{\bb Cov}}
\def\pr{{\bb P}}
\def\calW{\mathcal{W}}
\def\calN{\mathcal{N}}

\newcommand{\refs}[1]{(\ref{#1})}
\newtheorem{Theorem}{Theorem}[section]
\newtheorem{Proposition}[Theorem]{Proposition}
\newtheorem{Lemma}[Theorem]{Lemma}
\newtheorem{Corollary}[Theorem]{Corollary}
\newremark{Rem}{Remark}[section]
\newremark{Example}{Example}[section]
\makeatother

\begin{document}
\begin{frontmatter}

\title{Asymptotics of supremum distribution of a Gaussian process over
a Weibullian time}
\runtitle{Supremum of a Gaussian process over a Weibullian time}

\begin{aug}
\author{\fnms{Marek} \snm{Arendarczyk}\thanksref{e1}\ead[label=e1,mark]{marendar@math.uni.wroc.pl}} \and
\author{\fnms{Krzysztof} \snm{D{\normalfont{\c{E}}}bicki}\corref{}\thanksref{e2}\ead[label=e2,mark]{debicki@math.uni.wroc.pl}}
\runauthor{M. Arendarczyk and K. D\c{e}bicki}
\pdfauthor{Marek Arendarczyk, Krzysztof Debicki}
\address{Mathematical Institute, University of Wroc\l aw,
pl. Grunwaldzki 2/4, 50-384 Wroc\l aw, Poland.\\ E-mail: \printead*{e1,e2}}
\end{aug}

\received{\smonth{7} \syear{2009}}
\revised{\smonth{1} \syear{2010}}

%
\begin{abstract}
Let $\{X(t)\dvtx t\in[0,\infty)\}$ be a centered Gaussian process with
stationary increments and variance function $\sigma^2_X(t)$. We
study the exact asymptotics of
$\pr(\sup_{t\in[0,T]}X(t)>u)$ as $u\to\infty$, where $T$
is an independent of $\{X(t)\}$ non-negative Weibullian random
variable.
As an illustration, we work out the asymptotics of the supremum
distribution of
fractional Laplace motion.
\end{abstract}

%

\end{frontmatter}

\section{Introduction}
The problem of analyzing the asymptotic properties of
%
\begin{equation}\label{main}
\pr\Bigl(\sup_{t\in[0,T]}X(t)>u\Bigr) \qquad \mbox{as }  u\to\infty
\end{equation}
for a centered Gaussian process with stationary increments
$\{X(t)\}$ and deterministic $T>0$  plays an important role in many
fields of applied and theoretical probability.

One of the seminal results in this area is the exact asymptotic
%
\begin{equation}
\pr\Bigl(\sup_{t\in[0,T]}X(t)>u\Bigr)=\pr\bigl(X(T)>u\bigr)\bigl(1+\mathrm{o}(1)\bigr)\label{Ber}
\end{equation}
as $u\to\infty$, which holds for a wide class of centered Gaussian
processes (see \cite{Pit79,Tal88} and \cite{Pit96} for extensions of
this result).

Some recently studied problems in, for example, queueing theory (dual risk
theory) or hydrodynamics,
motivate the analysis of \refs{main}
for $T$ being a non-negative random variable independent of $\{X(t)\}$.
In particular,
the tail asymptotics of the \textit{steady-state buffer content} for a
\textit{hybrid fluid queue}
with the input modeled by a superposition of an integrated \textit{on-off
process} and a Gaussian process with stationary increments
can be reduced (under some assumptions) to the analysis of~\refs{main}
for some suitably chosen random $T$ (see, e.g., \cite{BDZ05} and
references therein).
Additionally, the analysis of the
supremum distribution of subordinated Gaussian processes is strongly
related to
\refs{main} over random $T$.
For example, the asymptotics of the supremum of a \textit{fractional
Laplace motion}
(used in hydrodynamic models -- see, e.g., \cite{Koz04,Koz06})
over a deterministic interval can be reduced to \refs{main} with
$X(t)$ being a fractional Brownian motion and
$T$ having Weibull distribution.
We refer to Section \ref{s.flm} for details.

We note that the additional variability of $T$ may influence
the form of the asymptotics of \refs{main}, leading to
structures qualitatively different from \refs{Ber}. This was observed in
\cite{BDZ04}, under the scenario that $T$ has a regularly varying tail
distribution
(see also \cite{Abu07}).

Motivated by the above applications,
in this paper, we focus on the exact asymptotics of \refs{main}
when $T$ is a random variable, independent of $\{X(t)\}$, with
asymptotically Weibullian tail distribution.
In Theorem
\ref{th.main}, we find the structural form of the asymptotics that
holds for a wide class of Gaussian processes
with stationary increments and convex variance function (see
assumptions~(A1)--(A3) in Section \ref{s.preliminaries}).
Complementing this, in Corollary \ref{cor.exact}, we obtain an
explicit form for the asymptotics, which appear to be Weibullian.

Additionally, for $\{X(t)\}$ being a fractional Brownian motion,
we provide the exact asymptotics of \refs{main} for the whole range of
Hurst parameters $H\in(0,1]$.
It appears that in the case of $H<1/2$ (concave variance function), the
exact asymptotics
takes a form qualitatively different from \refs{Ber}.

Finally, in Section \ref{s.flm}, we apply the obtained results to the
analysis of extremal behavior
of fractional Laplace motion; see \cite{Koz04,Koz06}.

\section{Notation and preliminary results}\label{s.preliminaries}
Let $\{X(t)\dvtx t\in[0,\infty)\}$ be a centered Gaussian process with
stationary increments, a.s. continuous sample paths, $X(0)=0$ a.s. and
variance function $\sigma^2_X(t):=\var(X(t))$. We assume that:
\begin{longlist}[(A3)]
\item[(A1)] $\sigma^2_X(\cdot)\in C^1([0,\infty))$ is convex;
\item[(A2)] $\sigma^2_X(\cdot)$ is regularly varying at $\infty$ with parameter
$\alpha_\infty\in(1 , 2)$;
\item[(A3)] there exists $D>0$ such that $\sigma^2_X(t)\le Dt^{\alpha
_\infty}$
for each $t\ge0$.
\end{longlist}

We introduce the following classes of Gaussian processes:
\begin{itemize}
\item \textbf{fBm}: $X(t)=B_H(t)$ is a fractional Brownian motion
with Hurst parameter
$H\in(0,1]$, that is, a centered Gaussian process with stationary
increments and
$\sigma^2_{B_H}(t)=t^{2H}$ (note that (A2) is satisfied for $H\in
(1/2,1)$);
\item \textbf{IG}: $X(t)=\int_0^t Z(s)\,\mathrm{d}s$, where $\{Z(t)\dvtx t\ge0\}$
is a centered stationary Gaussian process with covariance function
$R(t)=\cov(Z(s),Z(s+t))$
which is regularly varying at $\infty$ with parameter $\alpha_\infty-2$.
\end{itemize}

In this paper, we analyze the asymptotics of
%
\begin{equation}
\pr\Bigl(\sup_{t\in[0,T]}X(t)>u\Bigr)\label{randsup}
\end{equation}
as $u\to\infty$, where $T$ is a non-negative random variable,
independent of $\{X(t)\}$,
with asymptotically Weibullian tail distribution, that is,
%
\begin{equation}
\pr(T>t)= Ct^{\gamma}\exp(-\beta t^\alpha) \bigl(1 + \mathrm{o}(1)\bigr)\label{weib}
\end{equation}
as $t \to\infty$,
where $\alpha, \beta, C > 0, \gamma\in\R$.
We write
$T\in\mathcal{W} (\alpha,\beta,\gamma, C)$ if $T$ satisfies \refs{weib}.


Let us introduce some notation.
For given $H\in(0,1]$,
by $\mathcal{H}_H$, we denote the \textit{Pickands's constant} defined
by the limit
\[
\mathcal{H}_H = \lim_{T \to\infty} \frac{\mathcal{H}_H(T)}{T},
\]
where $\mathcal{H}_H(T):=\E\exp (\sup_{t \in[0,T]} \sqrt{2}
B_H(t) - t^{2H}  )$.
Moreover, let $\Psi(u):=\pr(\mathcal{N}>u)$, where $\mathcal{N}$
denotes the standard normal random variable.
$\dot{\sigma}_X(t)$ denotes the first derivative of $\sigma_X(t)$ and
$\dot{\sigma}^2_X(t)=2\sigma_X(t)\dot{\sigma}_X(t)$ the first
derivative of $\sigma^2_X(t)$.

Finally, we present a useful lemma,
which is also of independent interest.

\begin{Lemma} \label{l.product}
Let $X\in\calW(\alpha_1, \beta_1,\gamma_1, C_1)$,
$Y\in\calW(\alpha_2, \beta_2,\gamma_2, C_2)$ be independent
non-negative random variables.
Then
$X\cdot Y\in\calW(\alpha, \beta,\gamma, C)$ with
\begin{eqnarray*}
 \alpha&=& \frac{\alpha_1\alpha_2}{\alpha_1 + \alpha_2} ,\\
 \beta&=& \beta_1^{\alpha_2/(\alpha_1+\alpha_2)}
\beta_2^ {\alpha_1/(\alpha_1 + \alpha_2)}
 \biggl[
 \biggl(\frac{\alpha_1}{\alpha_2} \biggr)^ {\alpha_2/(\alpha_1 +
\alpha_2)} +
 \biggl( \frac{\alpha_2}{\alpha_1}  \biggr)^ {\alpha_1/(\alpha
_1+\alpha_2)}  \biggr] ,
\\
 \gamma&=& \frac{\alpha_1\alpha_2 + 2 \alpha_1\gamma_2 + 2\alpha
_2\gamma_1}{2 (\alpha_1 + \alpha_2 )},
\\
 C &=& \sqrt{2\uppi}C_1C_2\frac{1}{\sqrt{\alpha_1 + \alpha_2}} (
\alpha_1 \beta_1 )^ {(\alpha_2 - 2\gamma_1 + 2\gamma_2) /(2
(\alpha_1 + \alpha_2))} (\alpha_2 \beta_2)^ {(\alpha_1 -
2\gamma_2
+ 2\gamma_1)/(2(\alpha_1+\alpha_2))}.
\end{eqnarray*}
\end{Lemma}

The proof of Lemma \ref{l.product} is presented in Section \ref{s.lem}.

\section{Main results}\label{s.general}

In this section, we present the main results of the paper. We begin
with the structural
form of the analyzed asymptotics (Theorem \ref{th.main}), then we present
an explicit asymptotic expansion (Corollary \ref{cor.exact}).

%
\begin{Theorem}\label{th.main}
Let $X(t)$ be a centered Gaussian process with stationary increments
and variance function that satisfies (\textup{A1})--(\textup{A3}) and
$T\in\mathcal{W}(\alpha, \beta,\gamma, C)$ be a non-negative
random variable, independent of $\{X(t)\}$.
Then, as $u \to\infty$,
\[
\pr\Bigl(\sup_{s \in[0,T]} X(s) >u \Bigr)
=
\pr\bigl( X(T) >u \bigr) \bigl(1 + \mathrm{o}(1)\bigr)
=
\pr\bigl(\sigma_X(T)\cdot\calN> u\bigr)\bigl(1 + \mathrm{o}(1)\bigr).
\]
\end{Theorem}

The proof of Theorem \ref{th.main}
is presented in Section \ref{s.thmain}.

\begin{Rem}
It is tempting to ask to what extent
\refs{main} behaves as $\pr( X(T) >u )$ for other (than Weibullian)
distributions of $T$.
Some limitations on the heaviness of the tail distribution of $T$ can
be inferred from \cite{BDZ04}, Theorem 2.1,
which states that
%
\begin{equation}
\pr\Bigl(\sup_{s \in[0,T]} X(s) >u \Bigr)=\operatorname{Const} \pr\bigl(T>\sigma
_X^{-1}(u)\bigr)  \qquad \mbox{as }  u\to\infty, \label{regvar}
\end{equation}
if $T$ has regularly varying tail distribution at $\infty$.
Thus, the asymptotics of \refs{regvar} are qualitatively different
from those observed in Theorem \ref{th.main}.
We conjecture that an analog of Theorem \ref{th.main} is also true for
lighter-than-Weibullian tail distributions of $T$.
\end{Rem}

If the variance function of $\{X(t)\}$ is regular enough (in such a way
that $\sigma_X(T)$ is asymptotically Weibullian), then
the combination of Theorem \ref{th.main} with Lemma \ref{l.product}
enables us to obtain the exact form of the asymptotics.

\begin{Corollary}\label{cor.exact}
Let $X(t)$ be a centered Gaussian process with stationary increments
and variance function that satisfies (\textup{A1}) and $\sigma
_X^2(t)=Dt^{\alpha_\infty}+\mathrm{o}(t^{\alpha_\infty-\alpha})$ as $t\to
\infty$
for $\alpha_\infty\in(1,2)$ and $D>0$. If
$T\in\mathcal{W}(\alpha, \beta,\gamma, C)$ is a non-negative
random variable independent of $\{X(t)\}$, then
\[
\sup_{s \in[0,T]} X(s)\in\mathcal{W}(\widetilde{\alpha
},\widetilde{ \beta},\widetilde{\gamma},\widetilde{ C})
\]
with
\begin{eqnarray*}
\widetilde{\alpha}&=&\frac{2\alpha}{\alpha+\alpha_\infty},\qquad
\widetilde{ \beta}=
\beta^{\alpha_\infty/(\alpha+\alpha_\infty)} \biggl(\frac
{D}{2} \biggr)^{ \alpha/(\alpha+\alpha_\infty)}
\biggl (  \biggl(\frac{\alpha}{\alpha_\infty}  \biggr)^{ \alpha
_\infty/(\alpha+\alpha_\infty)} +
 \biggl(\frac{\alpha_\infty}{\alpha}  \biggr)^{\alpha
/(\alpha+\alpha_\infty)}  \biggr),\\
\widetilde{\gamma}&=&\frac{2\gamma}{\alpha+\alpha_\infty},\qquad
\widetilde{ C}=
CD^{-1/\alpha_\infty}\sqrt{\frac{\alpha_\infty}{2(\alpha+\alpha
_\infty)}} \biggl( \frac{\alpha_\infty}{2\alpha\beta} D^{\alpha
_\infty/\alpha}  \biggr)^{ \gamma/(\alpha+\alpha_\infty)}.
\end{eqnarray*}
\end{Corollary}

The proof of Corollary \ref{cor.exact} is given in Section \ref{s.thexact}.

Below, we apply the obtained asymptotics to IG processes. The family of
fBm is analyzed separately
in Section \ref{s.fbm}.
Due to the self-similarity of fBm, we are able to give a proof
(independent of Theorem \ref{th.main})
that covers the whole range of Hurst parameters $H\in(0,1]$.
\begin{Example}
Let $T\in\mathcal{W}(\alpha, \beta,\gamma, C)$ and
$X(t)=\int_0^t Z(s)\,\mathrm{d}s$, where $\{Z(s)\dvtx s\ge0\}$ is a centered
stationary Gaussian process with continuous covariance function $R(t)$
such that
$R(t)=Dt^{\alpha_\infty-2}+ \mathrm{o}(t^{\alpha_\infty-2-\alpha})$ as $t
\to\infty$ with $\alpha_\infty\in(1,2)$.
Following Karamata's theorem (see, e.g., \cite{Bin87}, Proposition~1.5.8),
\[
\sigma_X^2(t) = 2\int_0^t \,\mathrm{d}s \int_0^s R(v) \,\mathrm{d}v =
\frac{2D}{\alpha_\infty(\alpha_\infty- 1)} t^{\alpha_\infty}
+\mathrm{o}(t^{\alpha_\infty-\alpha})
\]
as $t \to\infty$. Hence, by Corollary \ref{cor.exact}, we have
\[
\sup_{t\in[0,T]}X(t)\in\mathcal{W}(\widetilde{\alpha},\widetilde
{ \beta},\widetilde{\gamma},\widetilde{ C})
\]
with
\begin{eqnarray*}
\widetilde{\alpha}&=&\frac{2\alpha}{\alpha+\alpha_\infty},\\
\widetilde{ \beta}&=&
\beta^{ \alpha_\infty/(\alpha+\alpha_\infty)} \biggl(\frac
{D}{\alpha_\infty(\alpha_\infty- 1)} \biggr)^{ \alpha
/(\alpha+\alpha_\infty)}
 \biggl(  \biggl(\frac{\alpha}{\alpha_\infty}  \biggr)^{ \alpha
_\infty/(\alpha+\alpha_\infty)} +
 \biggl(\frac{\alpha_\infty}{\alpha}  \biggr)^{ \alpha
/(\alpha+\alpha_\infty)}  \biggr),\\
\widetilde{\gamma}&=&\frac{2\gamma}{\alpha+\alpha_\infty},\\
\widetilde{ C}&=&
C \biggl( \frac{2D}{\alpha_\infty(\alpha_\infty- 1)}
\biggr)^{-1/\alpha_\infty}\sqrt{\frac{\alpha_\infty}{2(\alpha+\alpha
_\infty)}}
\biggl ( \frac{\alpha_\infty}{2\alpha\beta} \biggl (\frac
{2D}{\alpha_\infty(\alpha_\infty- 1)} \biggr)^{\alpha_\infty
/\alpha}  \biggr)^{ \gamma/(\alpha+\alpha_\infty)}.
\end{eqnarray*}
\end{Example}

\section{The case of fBm}\label{s.fbm}
In this section, we focus on the exact asymptotics of \refs{randsup}
for $\{X(t)\}$ being an fBm.
The self-similarity of fBm,
combined with Lemma \ref{l.product}, enables us to provide the
following theorem.

\begin{Theorem}\label{th.fbm}
Let $\{B_H(s) \dvtx  s \geq0\}$ be an fBm with Hurst parameter $H \in(0,1]$
and $T\in\calW(\alpha,\beta,\gamma, C)$
be a non-negative
random variable independent of $\{B_H(s) \dvtx  s \geq0\}$. If:
\begin{longlist}[(iii)]
\item[(i)] $H \in(0,1/2)$, then
\[
\sup_{s \in[0,T]} B_H(s) \in
\mathcal{W} \biggl(\frac{2\alpha}{2H + \alpha}, \beta_1 , \frac
{2\alpha- 3\alpha H + 2\gamma}{ \alpha+ 2H}, C_1
 \biggr);
\]
%
%
\item[(ii)] $H = 1/2$, then
\[
\sup_{s \in[0, T]} B_H(s) \in
\mathcal{W} \biggl(\frac{2\alpha}{2H + \alpha}, \beta_1 , \frac
{2\gamma}{\alpha+ 2H}, 2 C_2  \biggr);
\]
%
%
\item[(iii)] $H \in(1/2, 1]$, then
\[
\sup_{s \in[0, T]} B_H(s) \in
\mathcal{W} \biggl(\frac{2\alpha}{2H + \alpha}, \beta_1 , \frac
{2\gamma}{\alpha+ 2 H}, C_2  \biggr),
\]
%
where
\begin{eqnarray*}
\beta_1 &=&\beta^{ {2H/(2H + \alpha)}}
\biggl (\frac{1}{2} \biggl(\frac{\alpha}{H} \biggr)^ {2H/(2H +
\alpha)} +
 \biggl(\frac{\alpha}{H} \biggr)^{- {\alpha/(2H + \alpha)}}
 \biggr)\\
C_1 &=& \mathcal{H}_H \biggl(\frac{1}{2} \biggr)^ {1/(2H)}
\frac{C}{\sqrt{2H + \alpha}}
H^ { (\alpha+ 6H + 2\gamma- 2)/(2\alpha+ 4H)}
(\alpha\beta)^ {(1 - 2H - \gamma)/(\alpha+ 2H)}, \\
 C_2 &=& \frac{C\sqrt{H}}{\sqrt{\alpha+ 2H}}
 \biggl( \frac{H}{\alpha\beta} \biggr)^ {\gamma/(\alpha+ 2H)}.
\end{eqnarray*}
\end{longlist}
\end{Theorem}

The following lemma plays an important role in the proof of Theorem
\ref{th.fbm}.

\begin{Lemma}\label{l.det.fbm}
Let $B_H(\cdot)$ be an fBm with Hurst parameter $H\in(0,1]$. If:
\begin{longlist}[(iii)]
\item[(i)] $H \in(0, 1/2)$, then
\[
\sup_{t\in[0,1]}B_H(t) \in
\mathcal{W} \biggl(2, \frac{1}{2}, \frac{1}{H} - 3, \frac{1}{H\sqrt
{\uppi}} 2^{- {(H+1)/(2H)}} \biggr);
\]
%
\item[(ii)] $H=1/2$, then
\[
\sup_{t\in[0,1]}B_H(t) \in
\mathcal{W} \biggl(2, \frac{1}{2}, -1, \frac{2}{\sqrt{2\uppi}}  \biggr);
\]
%
\item[(iii)] $H\in(1/2,1]$, then
\[
\sup_{t\in[0,1]}B_H(t) \in
\mathcal{W} \biggl(2, \frac{1}{2}, -1, \frac{1}{\sqrt{2\uppi}}  \biggr).
\]
%
\end{longlist}
\end{Lemma}

The proof of Lemma \ref{l.det.fbm} follows by a straightforward
application of \cite{Pit96}, Theorem D.3.

\begin{pf*}{Proof of Theorem \ref{th.fbm}}
Using the self-similarity of fBm, we have
\[
\pr\Bigl(\sup_{s \in[0, T]} B_H(s) > u\Bigr) = \pr\Bigl( T^H \sup_{s \in
[0,1]}B_H(s) > u \Bigr).
\]
Note that
$
T^H\in\mathcal{W}  (\frac{\alpha}{H}, \beta, \frac{\gamma
}{H}, C )
$
and (due to Lemma \ref{l.det.fbm})
$\sup_{s \in[0,1]} B_H(s)$ is asymptotically Weibullian.
%
%
%

Thus, all of the cases (i), (ii) and (iii) follow by a straightforward
application
of Lemma \ref{l.product}.

\mbox{}
\end{pf*}

\begin{Rem} \label{rem.brown}
Note that if $\pr(T>t)=\exp(-At)$, then for a standard Brownian
motion case, some straightforward calculations give
\[
\pr\Bigl(\sup_{t\in[0,T]}B_{1/2}(t)>u\Bigr)=\exp\bigl(-\sqrt{2A}u\bigr)
\]
for each $u\ge0$.
\end{Rem}


\section{Application to extremes of fractional Laplace
motion}\label{s.flm}
In this section, we apply Theorem \ref{th.fbm}
to the analysis of the asymptotics of the supremum distribution of
fractional Laplace motion over a deterministic interval.

Following \cite{Koz06}, we recall the definition of fractional Laplace
motion.

Let $\{\Gamma_t; t \geq0\}$ be a gamma process with parameter $\nu> 0$,
that is, a L\'evy process such that the increments
 $\Gamma_{t+s} - \Gamma_t$  have gamma distributions ${\cal
G}(s/\nu, 1)$ with density
\[
f(x) = \frac{1}{\Gamma(s/\nu)}
x^{s/\nu -1}\exp(-x),
\]
where $\Gamma(\cdot)$ denotes the gamma function.

Then, by fractional Laplace motion $\mathit{fLm}(\sigma, \nu)$, we denote the
process $\{ L_H(t); t \geq0 \}$ defined as follows:
\[
\{ L_H(t); t \geq0 \} \stackrel{d}{=} \{\sigma B_H(\Gamma_t); t \geq
0 \}.
\]
A standard fractional Laplace motion corresponds to $\sigma= \nu
= 1$ and is denoted by fLm. We refer to Kozubowski \textit{et
al.} \cite{Koz04,Koz06}
for motivations of interest in the analysis of this class of stochastic
processes.

Before we present the asymptotics of $\pr(\sup_{s \in[0,S]} L_H(s) >
u)$, let us observe that
for given $S>0$, we have
$\Gamma_S \in\calW (1, 1, S-1, \frac{1}{\Gamma(S)} ).$
Indeed, applying Karamata's theorem (see,
e.g., \cite{Bin87}, Proposition 1.5.10), we have
\[
\pr(\Gamma_S > u)
=
\frac{1}{\Gamma(S)} \int_u^\infty x^{S - 1} \mathrm{e}^{-x}\, \mathrm{d}x
=
\frac{1}{\Gamma(S)} \int_{\mathrm{e}^u}^\infty(\log y)^{S -1} y^{-2} \, \mathrm{d}y
=
\frac{1}{\Gamma(S)} u^{S-1} \mathrm{e}^{-u} \bigl(1 + \mathrm{o}(1)\bigr)
\]
as $u \to\infty$.

In the following proposition, we give the exact asymptotics of the
supremum of
fLm for $H > 1/2$.
Let
\[
m_H =  \biggl(\frac{1}{2} \biggr)^ {1/(2H + 1)}
\biggl [ \biggl(\frac{1}{2H} \biggr)^ {2H/(2H +1)} +
 \biggl(\frac{1}{2H} \biggr)^ {1/(2H+1)} \biggr].
\]
\begin{Proposition}\label{p.flm}
Let $L_H$ be a standard fLm. If $H > 1/2$, then
\[
\sup_{s \in[0,S]} L_H(s) \in
\mathcal{W}
 \biggl( \frac{2}{2H + 1}, m_H, \frac{2S - 2}{1 + 2H}, \frac{H^
{(S + 2H)/(2 + 4H)}}{\Gamma(S)\sqrt{1 + 2H}}  \biggr).
\]
%
\end{Proposition}

\begin{pf}
First, we consider the lower bound.
We observe that
\[
\pr\Bigl(\sup_{s \in[0,S]} B_H(\Gamma_s) > u\Bigr)
\geq
\pr\bigl( B_H(\Gamma_S) > u\bigr)
=
\pr\bigl((\Gamma_S)^H \mathcal{N} > u\bigr).
\]
Combining the above with the facts that $(\Gamma_S)^H\in\mathcal{W}
 (\frac{1}{H}, 1, \frac{S-1}{H}, \frac{1}{\Gamma(S)} )$ and
$\mathcal{N}\in\mathcal{W}  (2, \frac{1}{2}, -1,
\frac{1}{\sqrt{2\uppi}} )$, together with Lemma \ref
{l.product}, we obtain a tight asymptotic
lower bound.

We now focus on the upper bound.
Using the fact that sample paths of a gamma process are non-decreasing,
we get
\[
\pr\Bigl(\sup_{s \in[0,S]} B_H(\Gamma_s) > u\Bigr)
\leq
\pr\Bigl(\sup_{s \in[0,\Gamma_S]} B_H(s) > u\Bigr) .
\]
In order to complete the proof, it suffices to apply
(iii) of Theorem \ref{th.fbm}.
\end{pf}

\begin{Rem}
The case $H \le1/2$ should be handled with care. Applying the argument
presented in the proof of Proposition \ref{p.flm} gives
\begin{eqnarray*}
&&\pr\Bigl(\sup_{s \in[0,S]} L_H(s) > u\Bigr)
\\
&&\quad\geq
\frac{1}{\Gamma(S)\sqrt{1 + 2H}} H^ {(S + 2H)/(2 + 4H)}
u^{ {(2S - 2)/(1 + 2H)}}
\exp \bigl( - m_H
u^ {2/(2H + 1)}
 \bigr) \bigl(1 + \mathrm{o}(1)\bigr)
\end{eqnarray*}
as $u \to\infty$, and
\begin{eqnarray*}
\pr\Bigl(\sup_{s \in[0,S]} L_H(s) > u\Bigr)
&\leq&
\frac{1}{\Gamma(S)}
2^{- {1/(2H)}}
H^{- (2H + S +4)/(4H + 2)}
{\cal H}_H
u^ {(2SH - 4H + 1)/(H(2H + 1))}\\
&&{}\times\exp \bigl( - m_H
u^ {2/(2H + 1)}
 \bigr)\bigl (1 + \mathrm{o}(1)\bigr)
\end{eqnarray*}
as $u \to\infty$.
The above leads to the following logarithmic asymptotics
for $H \in(0, \frac{1}{2}]$:
\[
\frac{\log(\pr\sup_{s \in[0,S]} L_H(s) > u)}{u^ {2/(2H + 1)}}
= -m_H \bigl(1 + \mathrm{o}(1)\bigr)
\]
as $u \to\infty$.

In the case $H = \frac{1}{2}$, $S = 1$, due to Remark \ref
{rem.brown}, we
have
\[
\frac{1}{2}\exp\bigl(-\sqrt{2}u\bigr)
\leq
P\Bigl(\sup_{s \in[0,1]} L_{1/2}(s) > u\Bigr)
\leq
\exp\bigl(-\sqrt{2}u\bigr)
\]
for each $u \geq0$.
We conjecture that the exact asymptotics for $H\le1/2$ are influenced
by the distribution of jumps of the gamma process.
\end{Rem}

\section{Proofs}\label{s.proofs}

In this section, we present detailed proofs of Lemma \ref{l.product},
Theorem \ref{th.main}
and Corollary \ref{cor.exact}.

\subsection[Proof of Lemma 2.1]{Proof of Lemma \protect\ref{l.product}}\label{s.lem}
We begin by considering the asymptotic
\[
\int_{U(x_0(u))} f(x,u) \exp[S(x,u)] \, \mathrm{d}x
\]
as $u \to\infty$
for particular forms of $f(x,u)$ and $S(x,u)$, where
$x_0(u)$ denotes the point at which the function $S(x,u)$ of $x$
achieves its maximum over $[0,\infty)$ and
\[
U(x_0(u)) = \{x\dvtx  |x-x_0(u)| \leq
q(u)|S_{x,x}''(x_0(u),u)|^{-1/2}\}
\]
for some suitable
chosen function $q(u)$.
%
%

The following theorem can be found in, for example, \cite{Fed77},
Theorem 2.2.

\begin{Lemma}[(Fedoryouk)] \label{th.fed}
Suppose that there exists a function $q(u) \to\infty$ as $u \to
\infty$
such that
%
\begin{equation}
S''_{x,x}(x,u) = S''_{x,x}(x_0(u),u)[1+\mathrm{o}(1)] \label{W1Fed}
\end{equation}
and
%
\begin{equation}
f(x,u) = f(x_0(u),u)[1+\mathrm{o}(1)]  \label{W2Fed}
\end{equation}
as $u \to\infty$ uniformly for $x \in U(x_0(u))$.
Then
\[
\int_{U(x_0(u))} f(x,u) \exp[S(x,u)] \, \mathrm{d}x
=
\sqrt{-\frac{2\uppi}{S''_{x,x}(x_0(u),u)}} f(x_0(u),u)
\exp[S(x_0(u),u)]\bigl(1 + \mathrm{o}(1)\bigr)
\]
as $u \to\infty$.
\end{Lemma}

Lemma \ref{th.fed} enables us to get the following exact asymptotics,
which will play an important role in further analysis.


\begin{Lemma} \label{L1.Fed}
Let $\alpha_1$, $\alpha_2$, $\beta_1$, $\beta_2 > 0,$ $\gamma\in
\R$ and
$a(u) = u^ {\alpha_1/(2(\alpha_1 + \alpha_2))},$
$A(u) = u^ {2\alpha_1/(\alpha_1 + \alpha_2)}$.
Then
\[
\int_{a(u)}^{A(u)} x^\gamma\exp \biggl( {-\frac{\beta_1 u^{\alpha
_1}}{x^{\alpha_1}}} - \beta_2
x^{\alpha_2}  \biggr) \, \mathrm{d}x
=
C u^\delta
\exp [-\beta_3 u^{\alpha_3}  ]\bigl(1 + \mathrm{o}(1)\bigr)
\]
as $u \to\infty$,
where
%
\begin{eqnarray*}
 \alpha_3 &=& \frac{\alpha_1\alpha_2}{\alpha_1 + \alpha_2} ,\qquad
%
 \beta_3 = \beta_1^ {\alpha_2/(\alpha_1+\alpha_2)}
\beta_2^ {\alpha_1/(\alpha_1 + \alpha_2)}
 \biggl[  \biggl(\frac{\alpha_1}{\alpha_2} \biggr)^ {\alpha
_2/(\alpha_1 + \alpha_2)} +
 \biggl( \frac{\alpha_2}{\alpha_1}  \biggr)^ {\alpha_1/(\alpha
_1+\alpha_2)}
 \biggr] ,\\
 \delta&=& {\frac{\alpha_1(-\alpha_2 + 2\gamma+ 2)}{2(\alpha
_1+\alpha_2)} },\\
 C &=& \sqrt{2\uppi} \frac{1}{\sqrt{\alpha_1 + \alpha_2}}
(\alpha_1\beta_1)^ {(-\alpha_2 + 2\gamma+ 2)/(2(\alpha_1 +
\alpha_2))}
(\alpha_2\beta_2)^ {(-\alpha_1 - 2\gamma- 2)/(2(\alpha_1+\alpha_2))}.
\end{eqnarray*}
\end{Lemma}
\begin{pf}
Let
$x_0(u) =  ( \frac{\alpha_1 \beta_1}{\alpha_2 \beta_2 }
 )^{{1/(\alpha_1 + \alpha_2)}}
u^ {\alpha_1/(\alpha_1 + \alpha_2)}$,
$r(u) = u^{ {(1-\varepsilon)\alpha_1/(\alpha_1 + \alpha_2)}}$
for some $\varepsilon\in(0,\break \min(\alpha_2/2,1))$
and $\alpha_3,\beta_3, \delta, C$ be as in Lemma \ref{L1.Fed}.
It is convenient to decompose the analyzed integral in the following way:
\begin{eqnarray*}
&&\int_{a(u)}^{A(u)} x^\gamma\exp \biggl( {-\frac{\beta_1 u^{\alpha
_1}}{x^{\alpha_1}}} - \beta_2
x^{\alpha_2}  \biggr) \,\mathrm{d}x \\
&&\quad =
\int_{a(u)}^{x_0(u)-r(u)} + \int_{x_0(u)-r(u)}^{x_0(u)+r(u)} +\int
_{x_0(u)+r(u)}^{A(u)}=I_1+I_2+I_3.
\end{eqnarray*}

Applying Lemma \ref{th.fed}, we have, as $u\to\infty$,
%
\begin{eqnarray}\label{q1}
I_2= \int_{x_0(u)-r(u)}^{x_0(u)+r(u)} x^\gamma\exp \biggl( {-\frac
{\beta_1 u^{\alpha_1}}{x^{\alpha_1}}} - \beta_2
x^{\alpha_2}  \biggr) \,\mathrm{d}x
=
C u^\delta
\exp [-\beta_3 u^{\alpha_3}  ]\bigl(1 + \mathrm{o}(1)\bigr).
\end{eqnarray}

In order to complete the proof, it suffices to show that
$I_1,I_3=\mathrm{o}(I_2)$ as $u\to\infty$.
Since proofs for $I_1$ and $I_3$ are similar, we focus on the argument
that shows $I_1=\mathrm{o}(I_2)$ as $u\to\infty$.
Without loss of generality, we assume that $\gamma> 0$.
Then
\[
I_1
\leq
\bigl(x_0(u) - a(u)\bigr)^\gamma\bigl(x_0(u) - r(u) - a(u)\bigr)
\exp \biggl(-\frac{\beta_1 u^{\alpha_1}}{(x_0(u) - r(u))^{\alpha_1}}
- \beta_2 \bigl(x_0(u) - r(u)\bigr)^{\alpha_2}  \biggr),
\]
which, combined with the fact that (using a Taylor expansion)
%
\begin{eqnarray}
&&
-\frac{\beta_1 u^{\alpha_1}}{(x_0(u) - r(u))^{\alpha_1}}
- \beta_2 \bigl(x_0(u) - r(u)\bigr)^{\alpha_2}
\nonumber
\\
&&\quad =
%
 -\beta_3 u^{\alpha_3} - \frac{1}{2} (\alpha_1 + \alpha_2) (\alpha_1
\beta_1)^ {(\alpha_2 - 2)/(\alpha_1 + \alpha_2)} (\alpha_2
\beta_2)^ {(\alpha_1 + 2)/(\alpha_1 + \alpha_2)}\\
&&\hspace*{5pt}\qquad{}\times u^{\alpha_1(\alpha_2-2\varepsilon)/(\alpha_1 + \alpha_2)} \bigl(1 + \mathrm{o}(1)\bigr)\nonumber
\end{eqnarray}
as $u \to\infty$, straightforwardly implies that $I_1=\mathrm{o}(I_3)$ as
$u\to\infty$ (since $\varepsilon<\alpha_2/2$).
This completes the proof.
\end{pf}
%

\begin{pf*}{Proof of Lemma 2.1}
Let
$X\in\calW(\alpha_1, \beta_1, \gamma_1, C_1)$
and
$Y\in\calW(\alpha_2, \beta_2, \gamma_2, C_2)$
be independent non-negative random variables.
Define
$a(u) = u^ {\alpha_1/(2(\alpha_1 + \alpha_2))},$
$A(u) = u^ {2\alpha_1/(\alpha_1 + \alpha_2)}$
and consider the decomposition
\begin{eqnarray*}
\pr(XY > u) &=& \int_0^\infty\pr \biggl(X > \frac{u}{y} \biggr)\,\mathrm{d}F_Y(y) \\
& = &
\int_0^{a(u)} \pr\biggl (X > \frac{u}{y} \biggr)\,\mathrm{d}F_Y(y)
+
\int_{a(u)}^{A(u)} \pr \biggl(X > \frac{u}{y}
\biggr)\,\mathrm{d}F_Y(y)\\
&&{}+
\int_{A(u)}^\infty\pr \biggl(X > \frac{u}{y} \biggr)\,\mathrm{d}F_Y(y)\\
& = & I_1 + I_2 + I_3.
\end{eqnarray*}
We analyze each of the integrals $I_1, I_2, I_3$ separately.
In order to simplify notation, we introduce
$
h_1(u) = C_1 u^{\gamma_1} \exp ( -\beta_1 u^{\alpha_1}
)$ and
$h_2(u) = C_2 u^{\gamma_2} \exp ( -\beta_2 u^{\alpha_2} ).
$

%

\textit{Integral $I_1$.}
Since $X\in\calW(\alpha_1, \beta_1, \gamma_1, C_1)$,
for given $\varepsilon>0$ and $u$ large enough, we have
\begin{eqnarray*}\label{I_1}
I_1
&\leq&
(1 + \varepsilon) h_1\biggl (\frac{u}{a(u)} \biggr)\\
&=&
(1 + \varepsilon)C_1 u^{ {(\alpha_1+2\alpha_2)/(2(\alpha
_1+\alpha_2))}\gamma_1}
\exp \biggl( - \beta_1 u^{  {\alpha_1\alpha_2/(\alpha_1 +
\alpha_2)} +  {\alpha_1^2/(2(\alpha_1 + \alpha_2))}} \biggr).
\end{eqnarray*}


\textit{Integral $I_3$.}
For $u$ sufficiently large, we have, as $u \to\infty$,
\[ \label{I_3}
I_3
 \leq
\pr\bigl(Y > A(u)\bigr)
=
C_2 u^{ {2\alpha_1\gamma_2/(\alpha_1+\alpha_2)}}
\exp \bigl( - \beta_2
u^ {2 \alpha_1\alpha_2/(\alpha_1 + \alpha_2)} \bigr)\bigl(1 +
\mathrm{o}(1)\bigr).
\]


\textit{Integral $I_2$.}
We find upper and lower bounds of $I_2$ separately.
Using the fact that $X,Y$ are asymptotically Weibullian, we get,
for sufficiently large $u$,
\begin{eqnarray*}
&&\int_{a(u)}^{A(u)} \pr \biggl(X > \frac{u}{y} \biggr)\,\mathrm{d}F_Y(y)
\geq
(1 - \varepsilon)
\int_{a(u)}^{A(u)}
h_1 \biggl(\frac{u}{y} \biggr) \,\mathrm{d}F_Y(y)\\
&&\quad \geq
(1 - \varepsilon)
\int_{a(u)}^{A(u)}
\frac{\partial}{\partial y}  \Biggl[h_1 \biggl(\frac{u}{y} \biggr)
 \Biggr]
\pr(Y > y) \,\mathrm{d}y
+
(1 - \varepsilon)
h_1 \biggl(\frac{u}{a(u)} \biggr) \pr\bigl(Y > (a(u))\bigr)\\
&&\qquad {} -
(1 - \varepsilon) h_1  \biggl(\frac{u}{A(u)} \biggr)\pr\bigl(Y > A(u)\bigr)\\
&&\quad  \geq
(1 - \varepsilon)^2
\int_{a(u)}^{A(u)}
\frac{\partial}{\partial y}  \biggl[h_1 \biggl(\frac{u}{y} \biggr)
 \biggr]
h_2(y) \,\mathrm{d}y
+
(1 - \varepsilon)^2
h_1 \biggl(\frac{u}{a(u)} \biggr) h_2(a(u))\\
&&\qquad {}
-
(1 - \varepsilon^2) h_1  \biggl(\frac{u}{A(u)} \biggr)h_2(A(u))\\
&&\quad  =
(1 - \varepsilon)^2 I_4 + (1 - \varepsilon)^2 R_1 - (1 - \varepsilon^2)R_2.
\end{eqnarray*}
Analogously, for sufficiently large $u$, we have the upper bound
\[
I_2
\leq (1 + \varepsilon)^2 I_4 + (1 + \varepsilon)^2 R_1 - (1 -
\varepsilon^2)R_2.
\]

Additionally,
\begin{eqnarray*}\label{R_1}
R_1&=&
h_1 \biggl(\frac{u}{a(u)} \biggr)h_2(a(u))
\leq
h_1 \biggl(\frac{u}{a(u)} \biggr)\\
&=&
C_1u^ {(\alpha_1\gamma_1 + 2\alpha_2\gamma_1)/(2(\alpha_1 +
\alpha_2))}
\exp \bigl( -\beta_1 u^{ {\alpha_1\alpha_2/(\alpha_1 + \alpha
_2)}+ {\alpha_1^2/(2(\alpha_1 + \alpha_1))}}  \bigr)
\end{eqnarray*}
and
\[
R_2= h_1 \biggl(\frac{u}{A(u)} \biggr)h_2(A(u))
\leq
h_2(A(u))
=
C_2 u^ {2\alpha_1\gamma_2/(\alpha_1 + \alpha_2)}
\exp \bigl(-\beta_2 u^ {2\alpha_1 \alpha_2/(\alpha_1 + \alpha
_2)}  \bigr).
\]

Finally, applying Lemma \ref{L1.Fed}, we find the asymptotics of
integral $I_4$:
\[
I_4
=
C_3 u^{\gamma_3} \exp ( -\beta_3 u^{\alpha_3} ) \bigl(1 + \mathrm{o}(1)\bigr)
\]
as $u \to\infty$,
with
\begin{eqnarray*}
\alpha_3&=&\frac{\alpha_1\alpha_2}{\alpha_1 + \alpha_2},\qquad
\beta_3=\beta_1^ {\alpha_2/(\alpha_1+\alpha_2)}
\beta_2^ {\alpha_1/(\alpha_1 + \alpha_2)}
 \biggl[
 \biggl(\frac{\alpha_1}{\alpha_2} \biggr)^ {\alpha_2/(\alpha_1 +
\alpha_2)} +
 \biggl(
  \frac{\alpha_1}{\alpha_2}  \biggr)^ {\alpha_1/(\alpha
_1+\alpha_2)}  \biggr],
\\
\gamma_3 &=& \frac{\alpha_1\alpha_2 + 2 \alpha_1\gamma_2 + 2\alpha
_2\gamma_1}{2 (\alpha_1 + \alpha_2 )},\\
C_3 &=& \sqrt{2\uppi}C_1C_2\frac{1}{\sqrt{\alpha_1 + \alpha_2}} (
\alpha_1 \beta_1 )^ {(\alpha_2 - 2\gamma_1 + 2\gamma_2)/(2
(\alpha_1 + \alpha_2))} (\alpha_2 \beta_2)^ {(\alpha_1 -
2\gamma_2
+ 2\gamma_1)/(2(\alpha_1+\alpha_2))}.
\end{eqnarray*}
Since
$I_1,I_2,R_1,R_2 =\mathrm{o}((C_3 u^{\gamma_3} \exp( -\beta_3 u^{\alpha_3}))$
as $u\to\infty$, we have
\[
\pr(X\cdot Y>u)=I_4\bigl(1+\mathrm{o}(1)\bigr)=C_3 u^{\gamma_3} \exp ( -\beta_3
u^{\alpha_3} )\bigl(1+\mathrm{o}(1)\bigr)
\]
as $u\to\infty$. This
completes the proof.
\end{pf*}

%
\subsection[Proof of Theorem 3.1]{Proof of Theorem \protect\ref{th.main}}\label{s.thmain}
Let
$\tau_1 = \frac{2}{\alpha_\infty+ 2\alpha}$,
$\tau_2 = \frac{4}{2\alpha_\infty+ \alpha}$ and
$ \delta= \delta(u) = \frac{\sigma_X^3(t)}{\dot{\sigma}_X(t)} 2
u^{-2} \log^2u $.
Additionally, let $\{Z(s) \dvtx  s \geq0 \}$ be a centered stationary
Gaussian process
with covariance function
$
\cov(Z(s), Z(s+t)) = \mathrm{e}^{-t^{\alpha_\infty}}.
$
The existence of such a process is guaranteed by the fact that
$\alpha_\infty<2$, which implies that the covariance of $Z(\cdot)$ is
positively defined; see, for example, proof of \cite{Pit79}, Theorem D.3.

The proof of Theorem \ref{th.main} is based on the following two lemmas.

\begin{Lemma}\label{l.cut1}
Let $X(t)$ be a centered Gaussian process with stationary  increments
such that conditions (\textup{A1})--(\textup{A3}) are satisfied.
Then, for sufficiently large $u$, 
%
\[
\pr\Bigl(\sup_{s \in[0,t-\delta]} X(s) > u\Bigr)
\leq
\Psi   \biggl(\frac{u}{\sigma_X(t)} \biggr)
\exp\bigl(-\log^2(u)/2\bigr)
\]
uniformly for $t:=t(u)\in[u^{\tau_1},u^{\tau_2}]$.
\end{Lemma}
%
%
\begin{pf}
Let $t:=t(u)\in[u^{\tau_1},u^{\tau_2}]$. Observe that
$\sigma^2_X(t) 2 u^{-2} \log^2(u) \to0$ uniformly for
$t\in[u^{\tau_1},u^{\tau_2}]$ as $u\to\infty$ and $
\lim_{t \to\infty} \frac{\sigma_X(t)}{t\dot{\sigma}_X(t)}
$=
$ \lim_{t \to\infty}\frac{2\sigma_X^2(t)}{ t\dot{\sigma}_X^2(t)}
=
\frac{2}{\alpha_\infty}
$ (due to \cite{Bin87}, formula (1.11.1)). Hence,
%
\begin{equation} \label{delta}
\delta(u) = \mathrm{o}(t)\qquad \mbox{as } u \to\infty.
\end{equation}


Now, for sufficiently large $u$, we consider the following
decomposition:
\begin{eqnarray}\label{2_kawalki}
&&\pr\Bigl(\sup_{s \in[0 , t-\delta]} X(s) > u\Bigr)\nonumber\\
&&\quad \leq
\pr\Bigl(\sup_{s \in[0 , 1]} X(s) > u\Bigr)
\\
&&\qquad{}+
\sum_{k = 0}^{ ({D/\sigma_X^2(1)} )^
{1/\alpha_\infty}[t - \delta]}
\pr
\biggl( \sup_{s \in [1 +  ( {\sigma
_X^2(1)/D} )^ {1/\alpha_\infty} k ,
1 +  ( {\sigma_X^2(1)/D} )^ {1/\alpha_\infty
}(k + 1) ]}
\frac{X(s)}{\sigma_X(s)} > \frac{u}{\sigma_X(t - \delta)} \biggr).\nonumber
\end{eqnarray}
According to the Borell inequality (see, e.g., \cite{Adl90},
Theorem 2.1),
the first term
is bounded by
\[
\pr\Bigl(\sup_{s \in[0 , 1]} X(s) > u\Bigr)
\leq
\exp \biggl( - \frac{(u - \E\sup_{s \in[0 , 1]} X(s) )^2 }{2}
 \biggr)
\]
as $u \to\infty$.

Due to (A1), (A3), for each $v, w \geq1$ such that
$|v - w| \leq (\frac{\sigma^2_X(1)}{D} )^ {1/\alpha
_\infty}$,
\[
\cov \biggl( \frac{X(v)}{\sigma_X(v)},\frac{X(w)}{\sigma
_X(w)} \biggr)
\ge
\cov \biggl(Z  \biggl(  \biggl( \frac{D }{\sigma^2_X(1)} \biggr)^{
{1/\alpha_\infty}} v  \biggr) ,
Z \biggl( \biggl ( \frac{D }{\sigma^2_X(1)} \biggr)^{ {1/\alpha
_\infty}} w  \biggr)  \biggr).
\]
Thus, Slepian's inequality (see, e.g., \cite{Pit96}, Theorem C.1)
combined with \cite{Pit96}, Theorem D.2,
straightforwardly leads to
the following upper bound of \refs{2_kawalki}:
%
\begin{eqnarray}\label{A2.L1}
&& \biggl(\frac{D}{\sigma_X^2(1)} \biggr)^ {1/\alpha_\infty}[t -
\delta]
\pr \biggl( \sup_{s \in [0 ,  ( {\sigma
_X^2(1)/D} )^ {1/\alpha_\infty} ]}
Z( s) > \frac{u}{\sigma_X(t - \delta)} \biggr) \nonumber
\\[-8pt]
\\[-8pt]
&&\quad  =
\mathcal{H}_{\alpha_\infty}
t
\biggl ( \frac{u}{\sigma_X(t)}  \biggr)^ {2/\alpha_\infty}
\Psi \biggl( \frac{u}{\sigma_X(t-\delta)}  \biggr) \bigl(1 + \mathrm{o}(1)\bigr)\nonumber
\end{eqnarray}
as $u \to\infty$.
Hence, in order to complete the proof, it suffices to
note that
%
\begin{eqnarray}
\nonumber
\Psi \biggl( \frac{u}{\sigma_X(t - \delta)} \biggr)
& \leq&
4
\Psi \biggl(\frac{u}{\sigma_X(t)} \biggr)
\exp \biggl(-\frac{u^2( \sigma_X^2(t) - \sigma_X^2(t-\delta)
)}{2\sigma^2_X(t)\sigma_X^2(t-\delta)}  \biggr)\\
\nonumber
& \leq&
4
\Psi \biggl(\frac{u}{\sigma_X(t)} \biggr)
\exp \biggl( -\frac{u^2 \delta2\sigma_X(t - \theta\delta) \dot
{\sigma}_X(t - \theta\delta)}{2\sigma_X^4(t)}  \biggr)\\
\label{A3.L1}
& \leq&
4
\Psi \biggl(\frac{u}{\sigma_X(t)} \biggr)
\exp\biggl ( -\frac{ u^2 \delta\sigma_X(t) \dot{\sigma
}_X(t)}{2\sigma^4(t)}  \biggr)\\
\label{last1}
& = &
4
\Psi \biggl(\frac{u}{\sigma_X(t)} \biggr)
\exp(-\log^2(u)),
\end{eqnarray}
where $\theta\in(0,1)$, and \refs{A3.L1} is a consequence of \refs
{delta} and of the fact that, by condition (A1),
$\dot{\sigma}^2_X = 2\sigma_X(t)\dot{\sigma}_X(t)$ is monotone and
(in view of \cite{Bin87}, formula (1.11.1)) regularly varying at $\infty$. 

Thus, combining \refs{2_kawalki} with \refs{A2.L1} and
\refs{last1}, for sufficiently large $u$,
\begin{eqnarray*}
\pr\Bigl(\sup_{s \in[0 , t-\delta]} X(s) > u\Bigr)
&\le&
4 \mathcal{H}_{\alpha_\infty}
t\cdot
 \biggl( \frac{u}{\sigma_X(t)}  \biggr)^ {2/\alpha_\infty}
\Psi   \biggl(\frac{u}{\sigma_X(t)} \biggr)
\exp\bigl(-\log^2(u)\bigr)\bigl(1 + \mathrm{o}(1)\bigr)\\
&\le&
\Psi  \biggl (\frac{u}{\sigma_X(t)} \biggr)
\exp\bigl(-\log^2(u)/2\bigr),
\end{eqnarray*}
uniformly for $t\in[u^{\tau_1},u^{\tau_2}]$.
This completes the proof.
\end{pf}
%

\begin{Lemma} \label{l.as}
Let $X(t)$ be a centered Gaussian process with stationary  increments
such that conditions  (\textup{A1})--(\textup{A3}) are satisfied.
Then, for sufficiently large $u$,
\[
\pr\Bigl(\sup_{s \in[t-\delta, t]} X(s) > u\Bigr) \leq(1 + \varepsilon)
\Psi \biggl( \frac{u}{\sigma_X(t)}  \biggr)
\]
%
uniformly for $t:=t(u)\in[u^{\tau_1},u^{\tau_2}]$.
\end{Lemma}
\begin{pf}
Let $\varepsilon> 0$. Then
\[
\pr\Bigl(\sup_{s \in[t-\delta, t]} X(s) > u\Bigr)
 \leq
\pr\biggl (\sup_{s \in[t -
\delta, t]} \frac{X(s)}{\sigma_X(s)} > \frac{u}{\sigma_X(t)} \biggr).
\]
Using the fact that
for each $v, w \in[t - \delta, t]$,
\[
\cov \biggl( \frac{X(v)}{\sigma_X(v)},\frac{X(w)}{\sigma
_X(w)} \biggr)\ge
\cov \biggl(Z \biggl ( \biggl ( \frac{2D }{\sigma^2_X(t)}
\biggr)^{ {1/\alpha_\infty}} v  \biggr) ,
Z \biggl( \biggl(\frac{2D}{ \sigma_X^2(t)} \biggr)^{ {1/\alpha
_\infty}} w  \biggr)  \biggr),
\]
Slepian's inequality gives
%
\begin{eqnarray}\label{nowe1}
&&\pr \biggl(\sup_{s \in[t -
\delta, t]} \frac{X(s)}{\sigma_X(s)} > \frac{u}{\sigma_X(t)}
\biggr)\nonumber\\
&&\quad \leq
\pr \biggl(\sup_{s \in[t - \delta, t]}
Z \biggl( \biggl (\frac{2D}{\sigma_X^2(t)}  \biggr)^ {1/\alpha
_\infty} s \biggr) > \frac{u}{\sigma_X(t)} \biggr)
\\
&&\quad  =
\pr \biggl( \sup_{s \in [0  ,  (2D)^ {1/\alpha_\infty
} u^{ {2/\alpha_\infty}}
\delta(\sigma_X(t))^{- {4/\alpha_\infty}}  (
{u/\sigma_X(t)} )^{- {2/\alpha_\infty}}  ] } Z(s)
>  \frac{u}{\sigma_X(t)}  \biggr).\nonumber
\end{eqnarray}
Observe that for each $\varepsilon_1 > 0$, there exists $u$ large
enough such that
$(2D)^{{1}/{\alpha_\infty}} u^{{2}/{\alpha_\infty}} \times\break\delta
(\sigma_X(t))^{-{4}/{\alpha_\infty}}
\le\varepsilon_1$ uniformly for $t \in[u^{\tau_1}, u^{\tau_2}]$,
which, combined with \cite{Pit96}, Theorem D.1, implies
the following upper bound for \refs{nowe1}:
\begin{eqnarray*}
\pr \biggl( \sup_{s \in [0  ,  \varepsilon_1  (
{u/\sigma_X(t)} )^{- {2/\alpha_\infty}}  ] } Z(s)
> \frac{u}{\sigma_X(t)}  \biggr)
\le
(1 + \varepsilon_1) \mathcal{H}_{\alpha_\infty}(\varepsilon_1)\Psi
 \biggl(\frac{u}{\sigma_X(t)} \biggr)
\leq
(1 + \varepsilon) \Psi \biggl(\frac{u}{\sigma_X(t)} \biggr),
\end{eqnarray*}
where the last inequality is due to the fact that
$\mathcal{H}_{\alpha_\infty}(t)\to1$ as $t\to0$.

This completes the proof.
\end{pf}
%
\begin{pf*}{Proof of Theorem \ref{th.main}}
Since the lower bound is immediate, we focus on the analysis of the
upper bound.
We have
\begin{eqnarray*}
&&\pr\Bigl(\sup_{s \in[0, T]} X(s) > u\Bigr)\\
&&\quad  \leq
\int_0^{u^{\tau_1}} \pr\Bigl(\sup_{s \in[0, t]} X(s) > u\Bigr)\, \mathrm{d}F_T(t)  +
\int_{u^{\tau_1}}^{u^{\tau_2}} \pr\Bigl(\sup_{s \in[0, t-\delta]}
X(s) > u\Bigr) \, \mathrm{d}F_T(t) \\
&&\qquad {} +
\int_{u^{\tau_1}}^{u^{\tau_2}} \pr\Bigl(\sup_{s \in[t-\delta, t]}
X(s) > u\Bigr) \, \mathrm{d}F_T(t)
 +
\int_{u^{\tau_2}}^\infty\pr\Bigl(\sup_{s \in[0, t]} X(s) > u\Bigr) \, \mathrm{d}F_T(t)
\\
&&\quad  =
I_1 + I_2 + I_3 + I_4 .
\end{eqnarray*}
We now investigate the asymptotic behavior of each of the integrals.

\textit{Integral $I_1$.}
\[
I_1
 \le
\pr\Bigl(\sup_{s \in[0,1]} X(s) > u\Bigr) + \pr\Bigl(\sup_{s \in[1, u^{\tau
_1}]} X(s) > u\Bigr).
\]

Following an argument analogous to that given in the proof of Lemma
\ref{l.cut1}, we obtain the following asymptotic
upper bound for the above sum:
%
\begin{equation}\label{b.i1}
\operatorname{Const}u^{\tau_1}  \biggl( \frac{u}{\sigma_X(u^{\tau_1})}
\biggr)^ {2/\alpha_\infty}
\Psi \biggl(\frac{u}{\sigma(u^{\tau_1})} \biggr) \leq
\exp \bigl( - u^{ {2\alpha/(\alpha+ \alpha_\infty)
}+\varepsilon}  \bigr) \bigl(1 + \mathrm{o}(1)\bigr)
\end{equation}
as $u \to\infty$, for some $\varepsilon>0$.
%

\textit{Integral $I_2$.}
According to Lemma \ref{l.cut1}, for all $t \in[u^{\tau_1}, u^{\tau
_2}]$ and for $u$ large enough,
\[
\pr\Bigl(\sup_{s \in[0, t-\delta]} X(s) > u\Bigr) \leq\Psi   \biggl(\frac
{u}{\sigma_X(t)} \biggr)
\exp\bigl(- \log^2(u)/2\bigr).
\]
Hence,
%
\begin{eqnarray}\label{b.I2}
I_2
&\leq&
\exp\bigl(-\log^2(u)/2\bigr) \int_0^\infty\Psi   \biggl(\frac{u}{\sigma
_X(t)} \biggr) \, \mathrm{d}F_T(t)\nonumber
\\[-8pt]
\\[-8pt]
&=&\exp\bigl(-\log^2(u)/2\bigr)\pr\bigl(X(T)>u\bigr)=\mathrm{o}\bigl(\pr\bigl(X(T)>u\bigr)\bigr).\nonumber
\end{eqnarray}
%

\textit{Integral $I_3$.}
Due to Lemma \ref{l.as}, for each $\varepsilon> 0$ and $u$ large enough,
%
\begin{equation}\label{b.i3}
I_3  \leq
(1 + \varepsilon) \int_0^\infty\psi \biggl(\frac{u}{\sigma
_X(t)} \biggr) \, \mathrm{d}F_T(t)
%
=
(1 + \varepsilon) \pr\bigl( X(T) > u\bigr) .
\end{equation}
%

\textit{Integral $I_4$.}
\[\label{b.i4}
I_4
=
\pr(T > u^{\tau_2} )
\le
\exp \bigl(- u^{ {2\alpha/(\alpha+ \alpha_\infty)}+\varepsilon
} \bigr)\bigl (1 + \mathrm{o}(1)\bigr)
\]
as $u \to\infty$, for some $\varepsilon>0$.

Observe that for each $\epsilon>0$ and sufficiently large $u$,
\begin{eqnarray*}
\pr\bigl(X(T)>u\bigr)
&=&\pr\bigl(\sigma_X(T)\mathcal{N}>u\bigr) 
\ge
\pr\bigl(\sigma_X(T)>u^{ {\alpha_\infty/(\alpha+\alpha_\infty)}}\bigr)
\pr\bigl(\mathcal{N}>u^{ {\alpha/(\alpha+\alpha_\infty)}}\bigr)\\
&\ge&\exp \bigl( -u^{ {2\alpha/(\alpha+\alpha_\infty)}+\epsilon
} \bigr) .
\end{eqnarray*}
Thus $I_1,I_2,I_4=\mathrm{o}(I_3)$ as $u\to\infty$, which, in view of \refs
{b.i3}, implies that
\[
\pr \Bigl( \sup_{t\in[0,T]}X(t)>u  \Bigr)\le(1+\varepsilon)\pr\bigl(X(T)>u\bigr)
\]
for each $\varepsilon>0$ and sufficiently large $u$.
This completes the proof.
\end{pf*}

\subsection[Proof of Corollary 3.2]{Proof of Corollary \protect\ref{cor.exact}}\label{s.thexact}

By a straightforward inspection, we observe that
$\sigma_X^2(t)$ satisfies (A1)--(A3). Thus, in view of Theorem \ref
{th.main},
we have
\[
\pr\Bigl(\sup_{s \in[0,T]} X(s) >u \Bigr)= \pr\bigl(\sigma_X(T)\cdot\calN> u\bigr)\bigl(1
+ \mathrm{o}(1)\bigr)
\]
as $u\to\infty$.
Since $\mathcal{N}\in\mathcal{W}(2,1/2,-1,1/\sqrt{2\uppi})$, due to
Lemma \ref{l.product},
in order to complete the proof, it suffices to show that
$\sigma_X(T)\in\mathcal{W} (\frac{2\alpha}{\alpha_\infty
},\beta D^{\alpha/\alpha_\infty},\frac{2\gamma}{\alpha_\infty
},CD^{-1/\alpha_\infty}  ) $.
In view of
\[
\pr\bigl(\sigma_X(T)>u\bigr)
= C ((\sigma_X)^{-1}(u))^\gamma\exp ( -\beta((\sigma
_X)^{-1}(u))^\alpha )\bigl(1+\mathrm{o}(1)\bigr)
\]
and the fact that $(\sigma_X)^{-1}(u)=D^{-1/\alpha_\infty
}u^{2/\alpha_\infty}(1+\mathrm{o}(1))$, this
reduces to
\[\label{exp.1}
\exp ( -\beta((\sigma_X)^{-1}(u))^\alpha )=\exp \biggl(
-\frac{\beta}{D^{\alpha/\alpha_\infty} }u^{2\alpha/\alpha_\infty
}  \biggr)\bigl(1+\mathrm{o}(1)\bigr)
\]
as $u\to\infty$, which follows by inspection.

\section*{Acknowledgements}
Krzysztof D\c{e}bicki was partially supported by KBN Grant No. N
N2014079 33 (2007--2009) and by a Marie Curie Transfer of
Knowledge Fellowship of the European Community's Sixth Framework
Programme under contract number MTKD-CT-2004-013389.

\printhistory


\begin{thebibliography}{00}
%
\bibitem{Abu07}
Abundo, M. (2008). Some remarks on the maximum of a one-dimensional
diffusion process. \textit{Probab.
Math. Statist.} \textbf{28} 107--120.
\MR{2445506}
%
%
\bibitem{Adl90}
Adler, R.J. (1990). An introduction to continuity, extrema, and related
topics for general Gaussian
processes.
In
\textit{Inst. Math. Statist. Lecture Notes Monograph Series} \textbf
{12}. Hayward, CA: Inst. Math. Statist.
\MR{1088478}
%
\bibitem{Bin87}
Bingham, N.H., Goldie, C.M. and Teugels, J.L. (1987).
\textit{Regular Variation}. Cambridge: Cambridge Univ. Press.
\MR{0898871}
%
\bibitem{BDZ04}
Borst, S.C., D\c{e}bicki, K. and Zwart, A.P. (2004). The supremum of a
Gaussian process over a random interval.
\textit{Statist. Probab. Lett.} \textbf{68} 221--234.
\MR{2083491}
%
\bibitem{Fed77}
Fedoryuk, M. (1977). \textit{The Saddle-Point Method}. Moscow: Nauka.
\MR{0507923}
%
\bibitem{Koz04}
Kozubowski, T.J., Meerschaert, M.M., Molz, F.J. and Lu, S. (2004).
Fractional Laplace model for hydraulic conductivity. \textit{Geophys.
Res. Lett.} \textbf{31} 1--4. L08501.
%
\bibitem{Koz06}
Kozubowski, T.J., Meerschaert, M.M. and Podg\'{o}rski, K. (2006). Fractional
Laplace motion. \textit{Adv. in Appl. Probab.} \textbf{38} 451--464.
\MR{2264952}
%
\bibitem{Pit79}
Piterbarg, V.I. and Prisyazhn'uk, V. (1978). Asymptotic behaviour of the
probability of a large excursion for a non-stationary Gaussian
process. \textit{Theory Probab. Math. Statist.} \textbf{18} 121--133.
%
\bibitem{Pit96}
Piterbarg, V.I. (1996). \textit{Asymptotic Methods in the Theory of
Gaussian Processes and Fields}.
\textit{Translations of Mathematical Monographs} \textbf{148}. Providence,
RI: Amer. Math. Soc.
\MR{1361884}
%
\bibitem{Tal88}
Talagrand, M. (1988). Small tails for the supremum of Gaussian process.
\textit{Ann. Inst. H. Poincare Probab. Statist.} \textbf{24} 307--315.
\MR{0953122}
\bibitem{BDZ05}
Zwart, A.P., Borst, S.C. and D\c{e}bicki, K. (2005). Subexponential
asymptotics of hybrid fluid and ruin models. \textit{Ann. Appl.
Probab.} \textbf{15} 500--517.
\MR{2115050}
\end{thebibliography}
\end{document}